\documentclass[amssymb,twoside,11pt]{article}
\usepackage{latexsym,amsmath,graphicx}
\usepackage[dvipsnames]{xcolor}
\usepackage{placeins}
\usepackage{amsfonts}
\usepackage{enumitem}
\newtheorem{theorem}{Theorem}[section]
\newtheorem{Lemma}[theorem]{{\bf Lemma}}

\newtheorem{rem}[theorem]{{\bf Remark}}

\newtheorem{definition}{Definition}[section]
\numberwithin{equation}{section}
\newenvironment{proof}{\indent{\em Proof:}}{\quad \hfill
$\Box$\vspace*{2ex}}

\font\Bbb=msbm10 at 12pt

\newcommand{\R}{\mbox{\Bbb R}}
\newcommand{\N}{\mbox{\Bbb N}}

\begin{document}
\setcounter{page}{1}
\begin{center}
\vspace{0.4cm} {\large{\bf On the Nonlinear $\Psi$-Hilfer Hybrid Fractional  Differential Equations}} \\
\vspace{0.5cm}
Kishor D. Kucche $^{1}$ \\
kdkucche@gmail.com \\

\vspace{0.35cm}
Ashwini D. Mali  $^{2}$\\
maliashwini144@gmail.com\\

\vspace{0.35cm}
$^{1,2}$ Department of Mathematics, Shivaji University, Kolhapur-416 004, Maharashtra, India.\\
\end{center}

\def\baselinestretch{1.0}\small\normalsize

\begin{abstract}
In this paper, we initially derive the equivalent fractional integral equation to $\Psi$-Hilfer hybrid fractional differential equations and through it, we prove the existence of a solution in the weighted space. The primary objective of the paper is to obtain estimates on $\Psi$-Hilfer derivative and utilize it to derive the hybrid fractional differential inequalities involving $\Psi$-Hilfer derivative. With the assistance of these fractional differential inequalities, we determine the existence of extremal solutions, comparison theorems and uniqueness of the solution.

\end{abstract}
\def\baselinestretch{1.5}
\allowdisplaybreaks
\section{Introduction}
The classical theory of fractional differential equations (FDEs) relating to inequalities and comparison results has developed by Lakshmikantham et al.\cite{V. Lak1, Lak1, V. Lak2} and utilized this theory to demonstrate the qualitative and quantitative properties of the solution of various classes of nonlinear FDEs.
On the other hand, Dhage and Lakshmikantham\cite{Dhage1} inducted the study of integer order hybrid nonlinear differential equations. They have proved an existence theorem and built differential inequalities. The obtained inequalities at that point used to investigate the existence of extremal solutions and a comparison result. On the lines of \cite{Dhage1},  Zhao et al.\cite{Zhao} have developed the theory of hybrid FDEs involving Riemann--Liouville (RL) fractional  derivative operator and acquired fundamental fractional differential inequalities, the existence of extremal solutions and comparison principle. Other related works on the hybrid FDEs may be found in \cite{Dhage2, Herzallah, HybAhmad1, HybAhmad2, HybAhmad3, Ferraoun, Sitho, Sun1, Caballero, Sun2, Mahmudov}.

The idea of the fractional derivative with respect to another function is introduced by Kilbas et al.\cite{Kilbas} in the sense of RL fractional derivative. On a comparable line, Almeida\cite{Almeida} presented $\Psi$-Caputo fractional derivative and investigated many fascinating properties of this operator. Hilfer\cite{Hilfer} presented fractional derivative operator $\mathcal{D}_{a^+}^{\mu ,\nu }(\cdot)$  with  two parameters  $\mu \in (n-1, n), n \in \mathbb{N}$  and  $\nu ~(0\leq \nu \leq 1)$. The Hilfer derivative $\mathcal{D}_{a^+}^{\mu ,\nu }(\cdot)$ unifies the theory of FDEs involving  RL fractional derivative ($\nu=0$) and Caputo fractional derivative ($\nu=1$). The calculus of Hilfer derivative and the analysis of nonlinear FDEs involving it very well may be found in \cite{Furati, S. Abbas}.

 In \cite{Vanterler1}, Sousa and Oliveira presented the Hilfer version of the fractional derivative with respect to another function which we refer in the present paper as $\Psi$-Hilfer fractional derivative. The importance of the $\Psi$-Hilfer fractional derivative lies in the fact that it incorporates several well recognized fractional derivative operators as it's particular cases, for example, RL\cite{Kilbas}, Caputo\cite{Kilbas}, $\Psi$-RL\cite{Kilbas},  Hadmard\cite{Kilbas},  Riesz\cite{Kilbas}, Erd$\acute{e}$ly-Kober\cite{Kilbas}, $\Psi$-Caputo\cite{Almeida}, Katugampola\cite{Katugampola}, Hilfer\cite{Hilfer}  and so forth. In this way various properties of solutions of FDEs involving different fractional derivative operators recorded in \cite{Vanterler1} can be analyzed under one roof with a single fractional derivative operator.  Investigation of initial and boundary value problems for $\Psi$-Hilfer  about   existence,   uniqueness, data dependence and Ulam-Hyers stabilities may be found in  \cite{jose1, jose2,vanterler3, j1,  Mali1, Mali2, Mali3, Kharade, MS Abdo1, MS Abdo2}.


Motivated by the work of \cite{Dhage1, Zhao},  in  the present paper, we consider the following $\Psi$-Hilfer hybrid FDEs of the form
 \begin{align}
  & ^H \mathcal{D}^{\mu,\,\nu\,;\, \Psi}_{0^+}\left[ \frac{y(t)}{ f(t, y(t))}\right] 
   = g(t, y(t)),~a.e. ~t \in  (0,\,T],  ~\label{eqq1}\\
   & \left( \Psi \left( t\right) -\Psi \left( 0\right) \right)^{1-\xi }y(t)|_{t=0}=y_{0} \in\R ,\label{eqq2}
 \end{align}
   where $0<\mu<1, ~0\leq\nu\leq 1,~ \xi=\mu+\nu(1-\mu)$, ~$^H \mathcal{D}^{\mu,\nu;\, \Psi}_{0^+}(\cdot)$ is the $\Psi$-Hilfer fractional derivative of order $\mu$ and type $\nu$, $f\in C(J \times \R  \,, \R\setminus\{0\}) $ is bounded, $J=[0,T]$ and   
   $
   g\in \mathfrak{C}(J \times \R  \,, \R)= \{ h ~|~ \text{the map}~\omega \to h(\tau,\omega)\; \text{is  continuous } \;\text{for each} ~\tau \;\text{and the map} ~\tau \to h(\tau,\omega) \\
   \;\text{is measurable} \;\text{for each} ~\omega \}.
    $

 We obtain, the equivalent fractional integral equation to the $\Psi$-Hilfer hybrid FDEs \eqref{eqq1}-\eqref{eqq2} and  establish the existence of solution in the weighted space 
   $ C_{1-\xi ;\, \Psi }(J,\,\R)$.  Our main objective here is to obtain estimates on $\Psi$-Hilfer derivative and utilize it to develop the hybrid fractional differential inequalities involving $\Psi$-Hilfer derivative. Using the fractional differential inequalities in the setting of $\Psi$-Hilfer derivative, we then derive the existence of extremal solutions, comparison results and uniqueness of the solution.
\begin{itemize}[topsep=0pt,itemsep=-1ex,partopsep=1ex,parsep=1ex]
\item For $ ~\nu=0, ~\Psi(t)=t ,~y_{0}=0 $, the obtained outcomes in the current paper incorporates the investigation of \cite{Zhao} relating to nonlinear hybrid FDEs of the form
\begin{align*}
  & ^{RL} \mathcal{D}^{\mu}_{0^+}\left[ \frac{y(t)}{ f(t, y(t))}\right] 
   = g(t, y(t)),~a.e. ~t \in  (0,\,T],\\
   & y(0)=0 .
 \end{align*} 
 \item For $\mu=1, ~\nu=1, ~\Psi(t)=t $, the obtained outcomes in the current paper incorporates the investigation of \cite{Dhage1} pertaining to nonlinear integer order hybrid differential equations of the form 
 \begin{align*}
   & \frac{d}{dt}\left[ \frac{y(t)}{ f(t, y(t))}\right] 
    = g(t, y(t)),~a.e. ~t \in  (0,\,T],\\
    & y(0)=y_{0} \in\R .
  \end{align*}
\item For $f=1$, the obtained results are applicable to nonlinear $\Psi$-Hilfer FDEs of the form
 \begin{align*}
  & ^H \mathcal{D}^{\mu,\,\nu\,;\, \Psi}_{0^+}y(t)
   = g(t, y(t)),~a.e. ~t \in  (0,\,T],  \\
   & \left( \Psi \left( t\right) -\Psi \left( 0\right) \right)^{1-\xi }y(t)|_{t=0}=y_{0} \in\R.
 \end{align*}
\item  For $f=1,$ ~$\nu=0$ (in this case $\xi=\mu$), \,$\Psi(t)=t $, the acquired results includes the study of nonlinear FDEs involving Riemann--Liouville derivative\cite{Lak3}
 \begin{align*}
  & ^{RL} \mathcal{D}^{\mu}_{0^+}y(t)
   = g(t, y(t)),~a.e. ~t \in  (0,\,T],  \\
   & [t^{1-\mu}y(t)]_{t=0}=y_{0} \in\R .
 \end{align*}
\item  For $f=1, ~\nu=1$ (in this case $\xi=1$), $\,\Psi(t)=t $, the acquired results includes the study of nonlinear FDEs involving Caputo derivative\cite{V. Lak1}
 \begin{align*}
  & ^{C} \mathcal{D}^{\mu}_{0^+}y(t)
   = g(t, y(t)),~a.e. ~t \in  (0,\,T], \\
   & y(0)=y_{0} \in\R.
 \end{align*}
\end{itemize}
We have referenced above just a couple of extraordinary cases.   Aside from these, for various selections of parameters $\mu$, $\nu$ and the function $\Psi$, the obtained outcomes in the current paper additionally hold for nonlinear hybrid FDEs with well known fractional derivative operators recorded in \cite{Vanterler1}, for example,  RL, Caputo,$\Psi$-RL, $\Psi$-Caputo, Hadmard, Katugampola, Riesz, Erd$\acute{e}$ly-Kober, Hilfer and so forth.

The plan of the paper is as follows:   In the section 2, we give some definitions and results which are useful to prove the main results. In section 3, we derive an equivalent integral equation(IE) to the problem \eqref{eqq1}-\eqref{eqq2} and establish existence of solution to it. In section 4, we obtain estimates on $\Psi$-Hilfer derivative. Section 5 deals with hybrid fractional differential inequalities involving $\Psi$-Hilfer derivative.  Section 6, contribute the study of maximal and minimal solutions. In section 7, we obtain comparison theorems. Finally, uniqueness result is proved in the section 8. 

\section{Preliminaries} \label{preliminaries}
 
Let $[a,b]$ $(0<a<b<\infty)$ be a finite interval and $\Psi\in C^{1}([a,b],\mathbb{R})$ be an increasing function such that $\Psi'(t)\neq 0$, $\forall~ t\in [a,b]$.  We consider the  weighted space 
\begin{equation*}
C_{1-\xi ;\, \Psi }\left[ a,b\right] =\left\{ h:\left( a,b\right]
\rightarrow \mathbb{R}~\big|~\left( \Psi \left( t\right) -\Psi \left(
a\right) \right) ^{1-\xi }h\left( t\right) \in C\left[ a,b\right]
\right\} ,\text{ }0< \xi \leq 1,
\end{equation*}
endowed with the norm
\begin{equation}\label{space1}
\left\Vert h\right\Vert _{C_{1-\xi ;\Psi }\left[ a,b\right] }=\underset{t\in \left[ a,b\right] 
}{\max }\left\vert \left( \Psi \left( t\right) -\Psi \left( a\right) \right)
^{1-\xi }h\left( t\right) \right\vert.
\end{equation}
The weighted space $C_{1-\xi \, ;\Psi }(\left[ a,b\right] ,\,\R)$ is a partially ordered Banach space  with the norm $\left\Vert \cdot\right\Vert _{C_{1-\xi ;\Psi }\left[ a,b\right] }$ and the partial ordering relation $\preceq$ defined by
$$
h_1\preceq h_2~  \text{if and only if}~ \left( \Psi \left( t\right) -\Psi \left( a\right) \right)
^{1-\xi }h_1\left( t\right)\leq  \left( \Psi \left( t\right) -\Psi \left( a\right) \right)
^{1-\xi }h_2\left( t\right),~ t\in \left[ a,b\right],$$ 
where $h_1, h_2\in C_{1-\xi \, ;\Psi }(\left[ a,b\right] ,\,\R)$.
Note that
\begin{enumerate}[topsep=0pt,itemsep=-1ex,partopsep=1ex,parsep=1ex]
\item [(i)] $
h_1\prec h_2~  \text{if and only if}~ \left( \Psi \left( t\right) -\Psi \left( a\right) \right)
^{1-\xi }h_1\left( t\right) <  \left( \Psi \left( t\right) -\Psi \left( a\right) \right)
^{1-\xi }h_2\left( t\right),~t\in \left[ a,b\right];$
\item [(ii)] $
h_1= h_2~  \text{if and only if}~ \left( \Psi \left( t\right) -\Psi \left( a\right) \right)
^{1-\xi }h_1\left( t\right) =  \left( \Psi \left( t\right) -\Psi \left( a\right) \right)
^{1-\xi }h_2\left( t\right),~t\in \left[ a,b\right];$
\item [(iii)] $
h_1\succ h_2~  \text{if and only if}~ \left( \Psi \left( t\right) -\Psi \left( a\right) \right)
^{1-\xi }h_1\left( t\right) >  \left( \Psi \left( t\right) -\Psi \left( a\right) \right)
^{1-\xi }h_2\left( t\right),~t\in \left[ a,b\right].$
\end{enumerate}

\begin{definition} [\cite{Kilbas}]
Let  $h$ be an integrable function defined on $[a,b]$. Then the $\Psi$-Riemann-Liouville fractional integral of order $\mu>0 ~(\mu \in \R)$ of the function $h$ is given by 
\begin{equation}\label{P1}
I_{a^+}^{\mu \, ;\Psi }h\left( t\right) =\frac{1}{\Gamma \left( \mu
\right) }\int_{a}^{t}\Psi ^{\prime }\left( s\right) \left( \Psi \left(
t\right) -\Psi \left( s\right) \right) ^{\mu -1}h\left( s\right) ds.
\end{equation}
\end{definition}

\begin{definition} [\cite{Vanterler1}]
The $\Psi$-Hilfer fractional derivative of a function $h$ of order $0<\mu<1$ and type $0\leq \nu \leq 1$, is defined by
$$^H \mathcal{D}^{\mu, \, \nu; \, \Psi}_{a^+}h(t)= I_{a^+}^{\nu ({1-\mu});\, \Psi} \left(\frac{1}{{\Psi}^{'}(t)}\frac{d}{dt}\right)I_{a^+}^{(1-\nu)(1-\mu);\, \Psi} h(t).$$
\end{definition}
\begin{Lemma}[\cite{Kilbas,Vanterler1}]\label{lema2} 
Let $\chi, \delta>0$ and $\rho>n$.  Then
\begin{enumerate}[topsep=0pt,itemsep=-1ex,partopsep=1ex,parsep=1ex]
\item [(i)] $\mathcal{I}_{a^+}^{\mu \, ;\,\Psi }\mathcal{I}_{a^+}^{\chi \, ;\,\Psi }h(t)=\mathcal{I}_{a^+}^{\mu+\chi \, ;\,\Psi }h(t)$.
\item [(ii)]$
\mathcal{I}_{a^+}^{\mu \, ;\,\Psi }\left( \Psi \left( t\right) -\Psi \left( a\right)
\right) ^{\delta -1}=\frac{\Gamma \left( \delta \right) }{\Gamma \left(
\mu +\delta \right) }\left( \Psi \left(t\right) -\Psi \left( a\right)
\right) ^{\mu +\delta -1}.
$
\item [(iii)]$^{H}\mathcal{D}_{a^+}^{\mu ,\,\nu \, ;\,\Psi }\left( \Psi \left(t\right) -\Psi \left( a\right)
\right) ^{\xi -1}=0.
$
\item [(iv)]$^{H}\mathcal{D}_{a+}^{\mu ,\,\nu\, ;\,\psi }\left( \psi \left(t\right) -\psi \left( a\right)
\right) ^{\rho -1}=\frac{\Gamma \left( \rho \right) }{\Gamma \left(
\rho -\alpha \right) }\left( \psi \left(t\right) -\psi \left( a\right)
\right) ^{\mu -\rho -1}.
$
\end{enumerate}
\end{Lemma}

\begin{Lemma}[\cite{Vanterler1}]\label{teo1} 
If $h\in C^{n}[a,b]$, $n-1<\mu<n$ and $0\leq \nu \leq 1$, then
\begin{enumerate}[topsep=0pt,itemsep=-1ex,partopsep=1ex,parsep=1ex]
\item [(i)]$
I_{a^+}^{\mu \, ;\Psi }\text{ }^{H}\mathcal{D}_{a^+}^{\mu ,\nu \, ;\Psi }h\left( t\right) =h\left(t\right) -\overset{n}{\underset{k=1}{\sum }}\frac{\left( \Psi \left( t\right) -\Psi \left( a\right) \right) ^{\xi -k}}{\Gamma \left( \xi -k+1\right) }h_{\Psi }^{\left[ n-k\right] }I_{a^+}^{\left( 1-\nu \right) \left( n-\mu \right) \, ;\Psi }h\left( a\right)$\\
where, $h_{\Psi }^{\left[ n-k\right] }h(t)=\left( \frac{1}{\Psi'(t)}\frac{d}{dt}\right)^{n-k}h(t)$.
\item [(ii)]$
^{H}\mathcal{D}_{a^+}^{\mu ,\nu \, ;\Psi }I_{a^+}^{\mu \, ;\Psi }h\left( t\right)
=h\left( t\right).
$
\end{enumerate}
\end{Lemma}
\begin{definition}[\cite{Kilbas, Diethelm1}]
Let $\eta >0\,(\,\eta \in \R) $. {Then the one parameter Mittag-Leffler function is defined as
$$E_{\eta}(z)=\sum_{k=0}^{\infty}\frac{z^k}{\Gamma(k \eta +1)}.$$}
\end{definition}
\begin{Lemma}[\cite{Dhage3}] \label{hyb}
Let  $S$ be a non-empty closed, convex and bounded    subset of  the   Banach algebra $X$  and let $A:X\rightarrow X $ and $B:S\rightarrow X $ be two operators such that
\begin{itemize}[topsep=0pt,itemsep=-1ex,partopsep=1ex,parsep=1ex]
\item [(a)] $A$ is Lipschitzian with a Lipschitz constant $\alpha$;
\item [(b)] $B$ is completely continuous;
\item [(c)] $y=AyBx \implies y\in S~ \text{for all} ~x\in S$ and
\item [(d)] $\alpha M <1 $ where $M=\sup\left\lbrace \left\|Bx \right\|: x\in  S \right\rbrace $.
\end{itemize}
Then the operator equation $y=Ay By$ has a solution in $S$.
\end{Lemma}
\section{Existence of solution}
\begin{Lemma}\label{lem41}
A  function $y\in C_{1-\xi ;\, \Psi }(J,\,\R)$ is the solution of  the Cauchy problem for hybrid FDEs \eqref{eqq1}-\eqref{eqq2} if and only if it is solution of the following hybrid fractional IE
\begin{align}\label{a1}
y(t)&=f(t,y(t))\left\lbrace  \frac{y_{0}}{f(0,y(0))}\left( \Psi \left( t\right) -\Psi \left( 0\right) \right)^{\xi-1 }+\mathcal{I}_{0^+}^{\mu\,;\, \Psi}g(t,y(t))\right\rbrace ,~t\in(0,T].
\end{align}
\end{Lemma}
\begin{proof}
Let  $y\in C_{1-\xi ;\, \Psi }(J,\,\R)$ be the solution of  the Cauchy problem for hybrid FDEs \eqref{eqq1}-\eqref{eqq2}. Operating $\Psi$-RL fractional integral $\mathcal{I}_{0^+}
^{\mu\,;\, \Psi}$ on the both sides of the equation \eqref{eqq1} and  using  Lemma \ref{teo1}(i), we obtain
$$
\frac{y(t)}{f(t,y(t))}-\frac{\left( \Psi \left( t\right) -\Psi \left( 0\right) \right)^{\xi-1 }}{\Gamma(\xi)}\left[\mathcal{I}_{0^+}^{1 -\xi  \, ;\Psi } \frac{y(t)}{f(t,y(t))}\right] _{t=0}=\mathcal{I}_{0^+}^{\mu \, ;\Psi }g(t,y(t)),\,t\in(0,T].
$$
The above equation can be written as
\begin{equation}\label{a2}
y(t)=f(t,y(t))\left\lbrace  \frac{C }{\Gamma(\xi)}\left( \Psi \left( t\right) -\Psi \left( 0\right) \right)^{\xi-1 }+\mathcal{I}_{0^+}^{\mu \, ;\Psi }g(t,y(t))\right\rbrace,~t\in (0, T],
\end{equation}
where
$$
C=\left[ \mathcal{I}_{0^+}^{1 -\xi  \, ;\Psi } \frac{y(t)}{f(t,y(t))}\right] _{t=0}.
$$
Next, we evaluate the value of $C $ using initial condition. Multiplying $\left( \Psi \left( t\right) -\Psi \left( 0\right) \right)^{1-\xi }$ on both sides of the equation \eqref{a2}, we get
$$
\left( \Psi \left( t\right) -\Psi \left( 0\right) \right)^{1-\xi }y(t)=\frac{C\, }{\Gamma(\xi)}f(t,y(t))+\left( \Psi \left( t\right) -\Psi \left( 0\right) \right)^{1-\xi}f(t,y(t))\,\mathcal{I}_{0^+}^{\mu \, ;\Psi }g(t,y(t)),~t\in J.
$$
Putting  $t=0$  in the above equation and using the initial condition \eqref{eqq2}, we obtain
$$
C=\frac{y_0\,\Gamma(\xi)}{f(0,y(0))}.
$$
Putting value of $C$ in the equation \eqref{a2}, we get
$$
y(t)=f(t,y(t))\left[\frac{y_0\,}{f(0,y(0))} \left( \Psi \left( t\right) -\Psi \left( 0\right) \right)^{\xi-1 }+\mathcal{I}_{0^+}^{\mu \, ;\Psi }g(t,y(t))\right],~t\in (0, T],
$$
which is required equivalent hybrid fractional IE to the hybrid FDEs \eqref{eqq1}-\eqref{eqq2}.

Conversely, let $y\in C_{1-\xi ;\, \Psi }(J,\,\R)$ is a solution of  the  hybrid IE \eqref{a1}. 
Then it can be written as
\begin{equation}\label{a3}
\frac{y(t)}{f(t,y(t))}=\frac{y_0\,}{f(0,y(0))} \left( \Psi \left( t\right) -\Psi \left( 0\right) \right)^{\xi-1 }+\mathcal{I}_{0^+}^{\mu \, ;\Psi }g(t,y(t)),~t\in(0, T].
\end{equation}
 Operating the $\Psi$-Hilfer derivative  $ ^H \mathcal{D}^{\mu,\,\nu\,;\, \Psi}_{0^+}$ on both sides of the equation \eqref{a3}  and using the Lemma \ref{lema2}(iii) and Lemma \ref{teo1}(ii),  we obtain
 \begin{align*}
 ^H \mathcal{D}^{\mu,\,\nu\,;\, \Psi}_{0^+}\left[ \frac{y(t)}{f(t,y(t))}\right] &=\frac{y_0\,}{f(0,y(0))}\,^H \mathcal{D}^{\mu,\,\nu\,;\, \Psi}_{0^+}\left( \Psi \left( t\right) -\Psi \left( 0\right) \right)^{\xi-1 } +\,^H \mathcal{D}^{\mu,\,\nu\,;\, \Psi}_{0^+}\mathcal{I}_{0^+}^{\mu \, ;\Psi }g(t,y(t))\\
 &=g(t,y(t)).
 \end{align*}
  It remains to verify the initial condition \eqref{eqq2}. From the equation \eqref{a3}, we have
 $$
\left( \Psi \left( t\right) -\Psi \left( 0\right) \right)^{1-\xi }\frac{y(t)}{f(t,y(t))}=\frac{y_0}{f(0,y(0))} +\left( \Psi \left( t\right) -\Psi \left( 0\right) \right)^{1-\xi }\mathcal{I}_{0^+}^{\mu \, ;\Psi }g(t,y(t)).
 $$
 At $t=0$, above equation reduces to   
 $$
 \left[\left( \Psi \left( t\right) -\Psi \left( 0\right) \right)^{1-\xi }\,y(t)\right]_{t=0}=y_0,
 $$
which gives the initial condition \eqref{eqq2}. This completes the proof. 
 \end{proof}
 
 To prove the existence of solution  to the hybrid FDEs \eqref{eqq1}-\eqref{eqq2}, we need the following hypotheses  on $f$ and $g$:
  \begin{enumerate}[topsep=0pt,itemsep=-1ex,partopsep=1ex,parsep=1ex]
  \item [(H1)] The function $f\in C\left(J\times\R, \R\setminus\{0\} \right) $ is bounded and satisfies the following conditions:
  \begin{enumerate}
  \item  [(i)]The mapping $v \rightarrow \frac{v}{f(t,v)}$ is increasing in $\R$ a.e.  $t\in (0, T]$ ;
  \item [(ii)] there exists $L>0$ such that
   $$
    \left| f(t,x)-f(t,y)\right| \leq L \left| x-y\right|,~  t\in J~\text{and}\,\, x,y\in \R.
    $$
  \end{enumerate}  
  \item [(H2)] The function $ g\in \mathfrak{C}(J \times \R  \,, \R)$  and there exists a function $h\in C(J,\,\R)$ such that
   $$
    \left| g(t,y)\right| \leq h(t),\,~a.e.\,\, t\in J~ \text{and}~  \, y\in \R.
   $$
   \end{enumerate}
 
 \title{Existence Theorem}
 \begin{theorem}\label{tha3.2}
 Assume that the  hypotheses {\normalfont(H1)-(H2)} hold. Then the hybrid FDEs  \eqref{eqq1}-\eqref{eqq2} has a solution $y\in C_{1-\xi ;\, \Psi }(J,\,\R)$ provided
\begin{equation}\label{aa1}
L\left\lbrace \left| \frac{y_0\,}{f(0,y(0))}\right| + \frac{\left\| h\right\|_\infty \left( \Psi \left( T\right) -\Psi \left( 0\right) \right)^{\mu }  }{\Gamma(\mu+1)}\right\rbrace <1.
\end{equation}
 \end{theorem}
   \begin{proof}
   Let $X:=\left( C_{1-\xi ;\, \Psi }(J,\,\R), \,\left\Vert \cdot\right\Vert _{C_{1-\xi ;\,\Psi }\left( J,\,\R\right) }\right) $.  Then $X$ is a Banach algebra with the product of vectors defined by $(xy)(t)=x(t)y(t),\,t\in J$. Define,
   $$
   S=\{x\in X: \left\Vert x\right\Vert _{C_{1-\xi ;\,\Psi }\left(J,\,\R\right)  }\leq R\},
   $$
   where 
   $$
   R=K\left\lbrace  \left| \frac{y_0\,}{f(0,y(0))}\right| + \frac{\left\| h\right\|_\infty \left( \Psi \left( T\right) -\Psi \left( 0\right) \right)^{\mu+1-\xi }  }{\Gamma(\mu+1)}\right\rbrace 
   $$
   and $K$ is bound on $f$.
   Clearly, $S$ is closed, convex and bounded subset of $X$.
   Define two operators  $A:X\rightarrow X $ and $B:S\rightarrow X $ by
   \begin{align*}
  Ay(t)&= f(t,y(t)), \,t\in J,\\
  By(t)&= \frac{y_0\,\left( \Psi \left( t\right) -\Psi \left( 0\right) \right)^{\xi-1 }}{f(0,y(0))} +\frac{1}{\Gamma \left( \mu \right) }\int_{0}^{t}\Psi'(s)(\Psi(t)-\Psi(s))^{\mu-1} g(s,y(s))\,ds, \, t\in (0,T].
\end{align*}
        Then, the hybrid IE \eqref{a1} is transformed into the following operator equation 
      $$
      y= Ay By,\, y\in X.
      $$
     We prove that the operators $A$ and $B$ satisfies all the conditions of Lemma \ref{hyb}. The proof is given in the following steps:\\
\textbf{Step 1:} $A:X\rightarrow X$  is Lipschitz operator.\\ 
Using the hypothesis (H1)(ii), we obtain
\begin{align*}
\left| \left( \Psi \left( t\right) -\Psi \left( 0\right) \right)^{1-\xi} \left(  Ax(t)-Ay(t) \right)   \right|
 &=\left| \left( \Psi \left( t\right) -\Psi \left( 0\right) \right)^{1-\xi }\left(  f(t,x(t))-f(t,y(t)) \right) \right| \\
 &\leq L\left|\left( \Psi \left( t\right) -\Psi \left( 0\right) \right)^{1-\xi } \left( x(t)-y(t)\right) \right|\\
 &\leq L \left\Vert x-y\right\Vert _{C_{1-\xi ;\,\Psi }\left(J,\,\R\right)  }.
\end{align*}
This gives,
\begin{equation}\label{A1}
\left\Vert Ax-Ay\right\Vert _{C_{1-\xi ;\,\Psi }\left(J,\,\R\right)  }\leq L \left\Vert x-y\right\Vert _{C_{1-\xi ;\,\Psi }\left(J,\,\R\right)  }. 
\end{equation}
\textbf{Step 2:} $B:S\rightarrow X$  is completely continuous.\\ 
(i) $B:S\rightarrow X$ is continuous.\\ 
Let  $y_n \rightarrow y$  in $S$. Then,
\begin{align*}
&\left\Vert By_n-By\right\Vert _{C_{1-\xi ;\,\Psi }\left(J,\,\R\right)  }\nonumber\\
&=\underset{t\in J 
}{\max }\left\vert \left( \Psi \left( t\right) -\Psi \left( 0\right) \right)
^{1-\xi }\left(   By_n(t)-By(t)\right)   \right\vert\nonumber\\
&\leq \underset{t\in J }{\max }\frac{\left( \Psi \left( t\right) -\Psi \left( 0\right) \right)
^{1-\xi }}{\Gamma \left( \mu\right) } \int_{0}^{t}\Psi'(s)(\Psi(t)-\Psi(s))^{\mu-1}\left\vert g(s,y_n(s))-g(s,y(s))\right\vert\,ds.
\end{align*}
By continuity of $g$ and Lebesgue dominated convergence theorem, from the above inequality,  we have
$$
\left\Vert By_n-By\right\Vert _{C_{1-\xi ;\,\Psi }\left(J,\,\R\right)  }\rightarrow 0 ~\text{ as}~ n\rightarrow\infty.
$$
This proves $B:S\rightarrow X$ is continuous.\\ 
(ii) $B(S)=\left\lbrace By: y\in S\right\rbrace $ is uniformly bounded.\\
Using hypothesis (H2), for any $y\in S$ and $t\in J$, we have
\begin{align*}
&\left\vert \left( \Psi \left( t\right) -\Psi \left( 0\right) \right)
^{1-\xi }By(t) \right\vert\\
 &\leq \left\vert \frac{y_0}{f(0,y(0))}\right\vert +\frac{\,\left( \Psi \left( t\right) -\Psi \left( 0\right) \right)^{1-\xi }}{\Gamma \left( \mu
       \right) }\int_{0}^{t}\Psi'(s)(\Psi(t)-\Psi(s))^{\mu-1} \left\vert g(s,y(s))\right\vert \,ds\\
&\leq \left\vert \frac{y_0}{f(0,y(0))}\right\vert +\frac{\,\left( \Psi \left( t\right) -\Psi \left( 0\right) \right)^{1-\xi }}{\Gamma \left( \mu
       \right) }\int_{0}^{t}\Psi'(s)(\Psi(t)-\Psi(s))^{\mu-1} h(s)\,ds\\
&\leq \left\vert \frac{y_0}{f(0,y(0))}\right\vert +\left\| h\right\|_\infty\,\left( \Psi \left( t\right) -\Psi \left( 0\right) \right)^{1-\xi }\frac{\left( \Psi \left( t\right) -\Psi \left( 0\right) \right)^{\mu }}{\Gamma \left( \mu+1     \right) }   \\    
&\leq \left\vert \frac{y_0}{f(0,y(0))}\right\vert +\frac{\left\| h\right\|_\infty\left( \Psi \left( T\right) -\Psi \left( 0\right) \right)^{\mu+1-\xi }}{\Gamma \left( \mu+1     \right) }. 
\end{align*}
Therefore, 
\begin{equation}\label{A2}
\left\Vert By\right\Vert _{C_{1-\xi ;\,\Psi }\left(J,\,\R\right)  }\leq \left\vert \frac{y_0}{f(0,y(0))}\right\vert +\frac{\left\| h\right\|_\infty\left( \Psi \left( T\right) -\Psi \left( 0\right) \right)^{\mu+1-\xi }}{\Gamma \left( \mu+1   \right) }.
\end{equation}
(iii) $B(S)$ is equicontinuous. \\
Let any $y\in S$ and $t_1, t_2\in J $ with $t_1<t_2$. Then using hypothesis (H2), we have
\begin{align*}
&\left\vert \left( \Psi \left( t_2\right) -\Psi \left( 0\right) \right)
^{1-\xi }By(t_2)-\left( \Psi \left( t_1\right) -\Psi \left( 0\right) \right) ^{1-\xi }By(t_1) \right\vert\\
&=\left\vert\left\lbrace \frac{y_0}{f(0,y(0))} +\frac{\left( \Psi \left( t_2\right) -\Psi \left( 0\right) \right)
^{1-\xi } }{\Gamma \left( \mu\right) }\int_{0}^{t_2}\Psi'(s)(\Psi(t_2)-\Psi(s))^{\mu-1} g(s,y(s))\,ds\right\rbrace\right.\\
&\left.-\left\lbrace \frac{y_0}{f(0,y(0))} +\frac{\left( \Psi \left( t_1\right) -\Psi \left( 0\right) \right)
^{1-\xi }}{\Gamma \left( \mu\right) }\int_{0}^{t_1}\Psi'(s)(\Psi(t_1)-\Psi(s))^{\mu-1} g(s,y(s))\,ds\right\rbrace\right\vert\\
&\leq 
\left\vert   \frac{\left( \Psi \left( t_2\right) -\Psi \left( 0\right) \right)
^{1-\xi } }{\Gamma \left( \mu\right) }\int_{0}^{t_2}\Psi'(s)(\Psi(t_2)-\Psi(s))^{\mu-1} \left| g(s,y(s))\right| \,ds\right.\\
&\left.\qquad-\frac{\left( \Psi \left( t_1\right) -\Psi \left( 0\right) \right)
^{1-\xi } }{\Gamma \left( \mu\right) }\int_{0}^{t_1}\Psi'(s)(\Psi(t_1)-\Psi(s))^{\mu-1} \left| g(s,y(s))\right|\,ds\right\vert\\
&\leq 
\left\vert   \frac{\left( \Psi \left( t_2\right) -\Psi \left( 0\right) \right)
^{1-\xi } }{\Gamma \left( \mu\right) }\int_{0}^{t_2}\Psi'(s)(\Psi(t_2)-\Psi(s))^{\mu-1} h(s)\,ds\right.\\
&\left.\qquad-\frac{\left( \Psi \left( t_1\right) -\Psi \left( 0\right) \right)
^{1-\xi } }{\Gamma \left( \mu\right) }\int_{0}^{t_1}\Psi'(s)(\Psi(t_1)-\Psi(s))^{\mu-1}h(s)\,ds\right\vert\\
&\leq 
\left\vert   \frac{\left( \Psi \left( t_2\right) -\Psi \left( 0\right) \right)
^{1-\xi } \left\| h\right\|_\infty}{\Gamma \left( \mu\right) }\int_{0}^{t_2}\Psi'(s)(\Psi(t_2)-\Psi(s))^{\mu-1} \,ds\right.\\
&\left.\qquad-\frac{\left( \Psi \left( t_1\right) -\Psi \left( 0\right) \right)
^{1-\xi }\left\| h\right\|_\infty }{\Gamma \left( \mu\right) }\int_{0}^{t_1}\Psi'(s)(\Psi(t_1)-\Psi(s))^{\mu-1}\,ds\right\vert\\
&= 
  \frac{\left\| h\right\|_\infty}{\Gamma \left( \mu+1\right) }\left\lbrace  (\Psi(t_2)-\Psi(0))^{\mu+1-\xi} - (\Psi(t_1)-\Psi(0))^{\mu+1-\xi} \right\rbrace.
 \end{align*}
By the  continuity of  $\Psi$, from the above inequality it follows that 
$$
\text{if}\left|t_1-t_2 \right|\to 0 ~\text{then}\left| \left( \Psi \left( t_2\right) -\Psi \left( 0\right) \right)
^{1-\xi }By(t_2)-\left( \Psi \left( t_1\right) -\Psi \left( 0\right) \right) ^{1-\xi }By(t_1)\right|\to 0.
$$ 
From the parts (ii) and (iii), it follows that $B(S)$ is uniformly bounded and  equicontinous set in $X$. Then by Arzel$\grave{a}$-Ascoli theorem, $B(S)$ is relatively compact. Therefore, $B:S\rightarrow X$ is a compact operator. Since $B:S\rightarrow X$ is the continuous and compact operator, it is completely continuous.\\
\textbf{Step 3:} For $y\in X,$~$y=AyBx \implies y\in S~ \text{for all} ~x\in S$.\\ 
Let any $y\in X$ and $x\in S$   such that $y=AyBx$. Using the hypothesis (H2) and boundedness of $f$, for any $t\in J$, we have
\begin{align*}
&\left\vert \left( \Psi \left( t\right) -\Psi \left( 0\right) \right)
^{1-\xi }y(t) \right\vert\\
&=\left\vert \left( \Psi \left( t\right) -\Psi \left( 0\right) \right)
^{1-\xi }Ay(t)Bx(t) \right\vert\\
&=\left\vert \left( \Psi \left( t\right) -\Psi \left( 0\right) \right)
^{1-\xi }f(t,y(t))\left\lbrace  \frac{y_{0}\,\left( \Psi \left( t\right) -\Psi \left( 0\right) \right)^{\xi-1 }}{f(0,x(0))}+\frac{1}{\Gamma \left( \mu
\right) }\int_{0}^{t}\Psi'(s)(\Psi(t)-\Psi(s))^{\mu-1} g(s,x(s))\,ds\right\rbrace \right\vert\\
&\leq\left\vert f(t,y(t))\right\vert \left\lbrace \left\vert  \frac{y_{0}}{f(0,x(0))}\right\vert+\frac{ \left( \Psi \left( t\right) -\Psi \left(0\right) \right)
^{1-\xi }}{\Gamma \left( \mu
\right) }\int_{0}^{t}\Psi'(s)(\Psi(t)-\Psi(s))^{\mu-1} \left\vert g(s,x(s))\right\vert\,ds\right\rbrace \\
&\leq K \left\lbrace \left\vert  \frac{y_{0}}{f(0,x(0))}\right\vert+\frac{ \left( \Psi \left( t\right) -\Psi \left( 0\right) \right)
^{1-\xi }}{\Gamma \left( \mu
\right) }\int_{0}^{t}\Psi'(s)(\Psi(t)-\Psi(s))^{\mu-1} h(s)\,ds\right\rbrace \\
&\leq K \left\lbrace \left\vert  \frac{y_{0}}{f(0,x(0))}\right\vert+\frac{ \left( \Psi \left( t\right) -\Psi \left( 0\right) \right)
^{\mu+1-\xi }\left\| h\right\|_\infty}{\Gamma \left( \mu+1
\right) }\right\rbrace . 
\end{align*}
This gives
$$
\left\Vert y\right\Vert _{C_{1-\xi ;\,\Psi }\left(J,\,\R\right)  }\leq K \left\lbrace \left\vert  \frac{y_{0}}{f(0,x(0))}\right\vert+\frac{ \left( \Psi \left(T\right) -\Psi \left( 0\right) \right)
^{\mu+1-\xi }\left\| h\right\|_\infty}{\Gamma \left( \mu+1
\right) }\right\rbrace=R.  
$$
This implies, $y\in S$.
\\
\textbf{Step 4:} To prove $\alpha M<1$ where $M=\sup\left\lbrace \left\|By \right\|_{C_{1-\xi ;\,\Psi }\left(J,\,\R\right)}: y\in  S \right\rbrace $.\\
From inequality \eqref{A2}, we have
\begin{align*}
M=\sup\left\lbrace \left\|By \right\|_{C_{1-\xi ;\,\Psi }\left(J,\,\R\right)}: y\in  S \right\rbrace
\leq  \left\vert \frac{y_0}{f(0,y(0))}\right\vert +\frac{\left\| h\right\|_\infty\left( \Psi \left( T\right) -\Psi \left( 0\right) \right)^{\mu+1-\xi }}{\Gamma \left( \mu+1     \right) } .
\end{align*}
From the inequality \eqref{A1}, we have $\alpha=L$. Therefore, using the condition \eqref{aa1}, we have 
$$
\alpha M\leq L  \left\lbrace \left\vert \frac{y_0}{f(0,y(0))}\right\vert +\frac{\left\| h\right\|_\infty\left( \Psi \left( T\right) -\Psi \left( 0\right) \right)^{\mu+1-\xi }}{\Gamma \left( \mu+1     \right) }\right\rbrace <1.
$$ 
From steps 1 to 4, it follows that all the conditions of Lemma \ref{hyb} are fulfilled. Consequently,  applying Lemma \ref{hyb}, the operator equation $y=AyBy$ has a solution in  $S$, which acts as a solution of hybrid FDEs \eqref{eqq1}-\eqref{eqq2}.
\end{proof}

\section{Estimates on $\Psi$-Hilfer derivative  }
\begin{theorem}\label{Lem41}
Let $m\in C_{1-\xi ;\, \Psi }(J,\,\R)$. Let   $t_1\in(0, T]$ be such that  $m(t_1)=0$ and $m(t)\leq 0,\,t\in (0, t_1)$.  Then,\,$^H \mathcal{D}^{\mu,\,\nu\,;\, \Psi}_{0^+}m(t_1)\geq 0$. 
\end{theorem}
\begin{proof}
Define
$$
M_\Psi (t)=\int_{0}^{t} \Psi'(s)\left( \Psi(t)-\Psi(s)\right)^{-\xi} m(s) ds,~t\in J.
$$
Let any $h>0$ such that $0<t_1-h<t_1$. Then
\begin{align*}
&M_\Psi (t_1)-M_\Psi (t_1-h)\nonumber\\
&=\int_{0}^{t_1} \Psi'(s)\left( \Psi(t_1)-\Psi(s)\right)^{-\xi} m(s) ds-\int_{0}^{t_1-h} \Psi'(s)\left( \Psi(t_1-h)-\Psi(s)\right)^{-\xi} m(s) ds\nonumber\\
&=\int_{0}^{t_1-h} \Psi'(s)\left\lbrace  \left( \Psi(t_1)-\Psi(s)\right)^{-\xi} -\left( \Psi(t_1-h)-\Psi(s)\right)^{-\xi}\right\rbrace   m(s) ds\nonumber\\
&\quad+\int_{t_1-h}^{t_1} \Psi'(s)\left( \Psi(t_1)-\Psi(s)\right)^{-\xi} m(s) ds.\nonumber
\end{align*}
Let 
\begin{align*}
&I_1=\int_{0}^{t_1-h} \Psi'(s)\left\lbrace  \left( \Psi(t_1)-\Psi(s)\right)^{-\xi} -\left( \Psi(t_1-h)-\Psi(s)\right)^{-\xi}\right\rbrace   m(s) ds\\
\end{align*}
and 
\begin{align*}
& I_2=\int_{t_1-h}^{t_1} \Psi'(s)\left( \Psi(t_1)-\Psi(s)\right)^{-\xi} m(s) ds.
\end{align*}
Then
\begin{equation}\label{A11}
M_\Psi (t_1)-M_\Psi (t_1-h)=I_1+I_2.
\end{equation}
Since $\xi>0$ and $\Psi$ is increasing function, we have
$$
0<\left( \Psi(t_1-h)-\Psi(s)\right)^\xi< \left( \Psi(t_1)-\Psi(s)\right)^\xi  ,\, \text{for}~ 0< s< t_1-h<t_1. 
$$
This gives
$$
 \left( \Psi(t_1)-\Psi(s)\right)^{-\xi} -\left( \Psi(t_1-h)-\Psi(s)\right)^{-\xi}<0,\, \text{for}~ 0< s< t_1-h<t_1. 
$$
This coupled with hypothesis $m(s)\leq 0, s\in(0, t_1)$  and the fact  $\Psi'(s)>0,\, s\in J,$ we have
$$
 \Psi'(s)\left\lbrace  \left( \Psi(t_1)-\Psi(s)\right)^{-\xi} -\left( \Psi(t_1-h)-\Psi(s)\right)^{-\xi}\right\rbrace m(s) \geq0,\,\text{for}~ 0< s< t_1-h<t_1. 
$$
This implies
 $I_1\geq 0$. Therefore, the equation \eqref{A11} reduces to
\begin{equation}\label{A13}
M_\Psi (t_1)-M_\Psi (t_1-h)\geq I_2.
\end{equation}
Since $\left( \Psi(t)-\Psi(0)\right)^{1-\xi}m(t)$ is continuous on $J$, corresponding to $t_1$, there exists constant $K(t_1)>0$   such that for $0<t_1-h<s<t_1,$
$$
-K(t_1)(t_1-s)\leq \left( \Psi(t_1)-\Psi(0)\right)^{1-\xi}m(t_1)-\left( \Psi(s)-\Psi(0)\right)^{1-\xi}m(s)\leq K(t_1)(t_1-s).
$$
But by hypothesis $m(t_1)=0$. Therefore, for $0<t_1-h<s<t_1$, we have
$$
-K(t_1)(t_1-s)\leq -\left( \Psi(s)-\Psi(0)\right)^{1-\xi}m(s)\leq K(t_1)(t_1-s).
$$
This gives
\begin{equation}\label{A14}
\left( \Psi(s)-\Psi(0)\right)^{1-\xi}m(s)\geq -K(t_1)(t_1-s),~0<t_1-h< s< t_1.
\end{equation}
Since 
\begin{equation}\label{A15}
 -K(t_1)(t_1-s)>  -K(t_1)h.
\end{equation}
From inequalities \eqref{A14} and \eqref{A15},  we have 
$$
\left( \Psi(s)-\Psi(0)\right)^{1-\xi}m(s)> -K(t_1)h, \,  \,\,0<t_1-h<s<t_1.
$$ 
This implies
\begin{equation}\label{Am15}
\Psi'(s)\left( \Psi(t_1)-\Psi(s)\right)^{-\xi}m(s)\geq -h K(t_1)\Psi'(s)\left( \Psi(t_1)-\Psi(s)\right)^{-\xi}\left( \Psi(s)-\Psi(0)\right)^{\xi-1},
\end{equation} for $ 0<t_1-h<s<t_1$.
By increasing nature of $\Psi$, we have
$$\Psi(t_1-h)<\Psi(s)<\Psi(t_1),\, s\in(t_1-h, t_1).$$
Since $\xi\leq 1,$ from above inequality, we have
\begin{equation}\label{Am16}
\left( \Psi(t_1-h)-\Psi(0)\right)^{\xi-1}>\left( \Psi(s)-\Psi(0)\right)^{\xi-1}>\left( \Psi(t_1)-\Psi(0)\right)^{\xi-1},~s\in(t_1-h, t_1).
\end{equation}
From the inequalities \eqref{Am15} and \eqref{Am16}, we have
$$
\Psi'(s)\left( \Psi(t_1)-\Psi(s)\right)^{-\xi}m(s)> -h K(t_1)\Psi'(s)\left( \Psi(t_1)-\Psi(s)\right)^{-\xi}\left( \Psi(t_1-h)-\Psi(0)\right)^{\xi-1}
$$
for $s\in(t_1-h, t_1).$ Integrating above inequality between $t_1-h$ to $t_1$ and using the definition of $I_2$, we obtain
\begin{align}\label{Aa15}
I_2
&> -h K(t_1) \left( \Psi(t_1-h)-\Psi(0)\right)^{\xi-1}\int_{t_1-h}^{t_1} \Psi'(s)\left( \Psi(t_1)-\Psi(s)\right)^{-\xi} ds\nonumber\\
&=-h K(t_1) \left(  \Psi(t_1-h)-\Psi(0)\right)^{\xi-1}\frac{\left( \Psi(t_1)-\Psi(t_1-h)\right)^{1-\xi}}{1-\xi}.
\end{align}
From the inequalities \eqref{A13} and \eqref{Aa15}, we have
$$
\frac{M_\Psi (t_1)-M_\Psi (t_1-h)}{h}> -K(t_1) \left(  \Psi(t_1-h)-\Psi(0)\right)^{\xi-1}\frac{\left( \Psi(t_1)-\Psi(t_1-h)\right)^{1-\xi}}{1-\xi}.
$$
Taking the limit as $h\rightarrow 0$ in the above inequality and using the continuity of $\Psi$, we obtain 
$$
\left[\frac{d}{dt} M_\Psi (t)\right]_{t=t_1}\geq 0.
$$
Since $\Psi'(t_1)>0$, using the definition of $M_\Psi$, we have
$$
\left[\frac{1}{\Psi'(t)} \frac{d}{dt}\left(\frac{1}{\Gamma(1-\xi)} \int_{0}^{t} \Psi'(s)\left( \Psi(t)-\Psi(s)\right)^{-\xi} m(s) ds\right) \right]_{t=t_1}\geq 0.
$$
This gives
$$
^H \mathcal{D}^{\mu,\,\nu\,;\, \Psi}_{0^+}m(t_1)=\left[I_{0^+}^{\nu \left(
1-\mu \right) \, ;\Psi }\left( \frac{1}{\Psi ^{\prime }\left( t\right) }\frac{d}{dt}\right) I_{0^+}^{\left( 1-\nu \right) \left( 1-\mu
\right) \, ;\Psi }m(t)  \right]_{t=t_1}\geq 0.
$$
\end{proof}

The dual of the Theorem \ref{Lem41} is also hold.
\begin{theorem}\label{Lem42}
Let $m\in C_{1-\xi ;\, \Psi }(J,\,\R)$. Let   $t_1\in(0, T]$ be such that  $m(t_1)=0$ and $m(t)\geq 0,\,t\in (0, t_1)$.  Then,\,$^H \mathcal{D}^{\mu,\,\nu\,;\, \Psi}_{0^+}m(t_1)\leq 0$. 
\end{theorem}
\begin{proof}
By utilizing the hypothesis on $m$, it follows that $(-m)$ fulfills the assumptions of Theorem \ref{Lem41} and the proof follows by applying it. 
\end{proof}
\section{Hybrid  differential inequalities with $\Psi$-Hilfer derivative}
\begin{theorem}\label{tha4.1}
Let $f\in C\left(J\times\R, \R\setminus\{0\} \right),\,  g\in \mathfrak{C}\left(J\times\R, \R \right) $ and assume that the hypothesis {\normalfont(H1)(i)} hold. Let $y,\,z\in C_{1-\xi ;\, \Psi }(J,\,\R)$ are such that 
\begin{align}\label{a4}
^H \mathcal{D}^{\mu,\,\nu\,;\, \Psi}_{0^+}\left[ \frac{y(t)}{ f(t, y(t))}\right] 
&\leq g(t, y(t)),~a.e. ~t \in  (0,\,T],  ~
\end{align}
\begin{align}\label{a5}
^H \mathcal{D}^{\mu,\,\nu\,;\, \Psi}_{0^+}\left[ \frac{z(t)}{ f(t, z(t))}\right] 
&\geq g(t, z(t)),~a.e. ~t \in  (0,\,T],
\end{align}
one of the inequalities being strict. Then 
$$\left( \Psi \left( t\right) -\Psi \left( 0\right) \right)^{1-\xi }y(t)|_{t=0}<\left( \Psi \left( t\right) -\Psi \left( 0\right) \right)^{1-\xi }z(t)|_{t=0}
$$ 
implies that 
$$y\prec z~\text{in}~ C_{1-\xi ;\, \Psi }(J,\,\R).$$
\end{theorem}
\begin{proof}
Assume that the conclusion of the theorem doesn't hold. Using the continuity of $\left( \Psi \left( t\right) -\Psi \left( 0\right) \right)
^{1-\xi }y(t)$ and $\left( \Psi \left( t\right) -\Psi \left( 0\right) \right)
^{1-\xi }z(t)$ on $J$,  there exists $t_1\in (0,\,T]$ such that 
$$
\left( \Psi \left( t_1\right) -\Psi \left( 0\right) \right)
^{1-\xi }y(t_1)=\left( \Psi \left( t_1\right) -\Psi \left( 0\right) \right)
^{1-\xi }z(t_1)
$$ 
and 
$$
\left( \Psi \left( t\right) -\Psi \left( 0\right) \right)
^{1-\xi }y(t)<\left( \Psi \left( t\right) -\Psi \left( 0\right) \right)
^{1-\xi }z(t),~\, \, t\in [0, t_1).
$$ 
This gives 
\begin{equation}\label{a6}
y(t_1)=z(t_1)~\text{and}~
y(t) < z(t),\,~ \,t\in (0, t_1).
\end{equation}
Define,
$$
Y(t)=\frac{y(t)}{ f(t, y(t))}~ ~\text{and}~ ~Z(t)=\frac{z(t)}{ f(t, z(t))},~t\in (0, T].
$$
Then $Y, Z\in C_{1-\xi ;\, \Psi }(J,\,\R)$.
Using the relations in  the equation \eqref{a6} and the hypothesis (H1)(i), we obtain 
$$
Y(t_1)=Z(t_1) ~\text{and} ~Y(t)\leq Z(t),~ \,t\in (0, t_1).
$$
Define
$m(t) = Y(t) - Z(t), ~t\in (0, T]$. Then, $m\in C_{1-\xi ;\, \Psi }(J,\,\R)$. Further, $t_1\in (0,\,T] $  such that  
$$
m(t_1) = 0 ~ \text{and}~ m(t)\leq 0, ~t\in(0,\, t_1).
$$ 
Therefore by Theorem \ref{Lem41}, we obtain
$$
^H \mathcal{D}^{\mu,\,\nu\,;\, \Psi}_{0^+}m(t_1)\geq 0.
$$
This implies
$$
^H \mathcal{D}^{\mu,\,\nu\,;\, \Psi}_{0^+}Y(t_1)\geq\, ^H \mathcal{D}^{\mu,\,\nu\,;\, \Psi}_{0^+}Z(t_1).
$$
By  inequality \eqref{a4}  and 
 assuming inequality \eqref{a5} is strict, we obtain
$$
g(t_1, y(t_1))\geq\, ^H \mathcal{D}^{\mu,\,\nu\,;\, \Psi}_{0^+}Y(t_1)\geq\, ^H \mathcal{D}^{\mu,\,\nu\,;\, \Psi}_{0^+}Z(t_1)>g(t_1, z(t_1)).
$$
This is contradicts to the fact $y(t_1)= z(t_1)$. Therefore, we must have 
$$
y\prec z~\text{in}~ C_{1-\xi ;\, \Psi }(J,\,\R).
$$
\end{proof}
\begin{theorem}
Assume that the conditions of Theorem \ref{tha4.1} hold with nonstrict inequalities
 \eqref{a4} and \eqref{a5}. Further, assume that there exists a real number $L>0$ such that
\begin{equation}\label{a11}
g(t, x_1)-g(t, x_2)\leq L\left( \frac{x_1}{f(t, x_1)}-\frac{x_2}{f(t, x_2)}\right),\,~a.e.\,\, t\in J,
\end{equation}
for all $x_1, x_2 \in \R$ with $x_1\geq x_2$. 
Then 
$$
\left( \Psi \left( t\right) -\Psi \left( 0\right) \right)^{1-\xi }y(t)|_{t=0}\leq\left( \Psi \left( t\right) -\Psi \left( 0\right) \right)^{1-\xi }z(t)|_{t=0}
$$ implies that 
$$y\preceq z~\text{in}~ C_{1-\xi ;\, \Psi }(J,\,\R).$$
\end{theorem}
\begin{proof}
Let any $\epsilon>0$. Set
\begin{equation}\label{a12}
\frac{z_\epsilon(t)}{f(t, z_\epsilon(t))}=\frac{z(t)}{f(t, z(t))}+\epsilon\, E_{\mu}\left( 2L\left(  \Psi \left( t\right) -\Psi \left( 0\right) \right)^{\mu}\right),~ t\in  (0, T],
\end{equation}
where $ E_{\mu}(\cdot) $ is the one parameter Mittag-Leffler function.
This implies that
$$
\frac{z_\epsilon(t)}{f(t, z_\epsilon(t))}>\frac{z(t)}{f(t, z(t))},~t\in  (0, T].
$$
By using the hypothesis (H1)(i), we have 
\begin{equation}\label{a13}
z_\epsilon(t)>z(t),~t\in  (0, T].
\end{equation}
Define 
$$
Z_\epsilon(t)=\frac{z_\epsilon(t)}{f(t, z_\epsilon(t))}~\text{and}~Z(t)=\frac{z(t)}{f(t, z(t))},~t\in  (0, T].
$$
Then, the equation \eqref{a12} takes the form
$$
Z_\epsilon(t)=Z(t)+\epsilon E_{\mu}\left( 2L\left(  \Psi \left( t\right) -\Psi \left( 0\right) \right)^{\mu}\right),~t\in  (0, T].
$$
Operating $\Psi$-Hilfer fractional derivative operator $^H \mathcal{D}^{\mu,\,\nu\,;\, \Psi}_{0^+}$ on both sides of the above equation and using the inequality \eqref{a5}, we obtain
\begin{align}\label{a14}
^H \mathcal{D}^{\mu,\,\nu\,;\, \Psi}_{0^+}Z_\epsilon(t)
&=\,^H \mathcal{D}^{\mu,\,\nu\,;\, \Psi}_{0^+}Z(t)+\epsilon\, ^H \mathcal{D}^{\mu,\,\nu\,;\, \Psi}_{0^+}E_{\mu}\left( 2L\left(  \Psi \left( t\right) -\Psi \left( 0\right) \right)^{\mu}\right)\nonumber\\
&\geq g(t, z(t))+\epsilon\, ^H \mathcal{D}^{\mu,\,\nu\,;\, \Psi}_{0^+}E_{\mu}\left( 2L\left(  \Psi \left( t\right) -\Psi \left( 0\right) \right)^{\mu}\right).
\end{align}
But by Lemma \ref{lema2}(iv), we obtain
\begin{align}\label{a15}
^H \mathcal{D}^{\mu,\,\nu\,;\, \Psi}_{0^+}E_{\mu}\left( 2L\left(  \Psi \left( t\right) -\Psi \left( 0\right) \right)^{\mu}\right)
&=\,^H \mathcal{D}^{\mu,\,\nu\,;\, \Psi}_{0^+}\left\lbrace \sum_{k=0}^{\infty}\frac{[2L\left(  \Psi \left( t\right) -\Psi \left( 0\right) \right)^{\mu}]^k}{\Gamma(k\mu+1)}\right\rbrace \nonumber \\
&=\sum_{k=1}^{\infty}\frac{2^kL^k}{\Gamma(k\mu+1)} \frac{\Gamma(k\mu+1)}{\Gamma(k\mu+1-\mu)}\left(  \Psi \left( t\right) -\Psi \left( 0\right) \right)^{k\mu-\mu}\nonumber\\
&=\sum_{k=0}^{\infty}\frac{2^{k+1}L^{k+1}}{\Gamma(k\mu+\mu+1-\mu)}\left(  \Psi \left( t\right) -\Psi \left( 0\right) \right)^{k\mu+\mu-\mu}\nonumber\\
&=2L E_{\mu}\left( 2L\left(  \Psi \left( t\right) -\Psi \left( 0\right) \right)^{\mu}\right).
\end{align}
Therefore the inequality \eqref{a14}, reduces to
\begin{equation}\label{aa15}
^H \mathcal{D}^{\mu,\,\nu\,;\, \Psi}_{0^+}Z_\epsilon(t)
\geq g(t, z(t))+2L\epsilon E_{\mu}\left( 2L\left(  \Psi \left( t\right) -\Psi \left( 0\right) \right)^{\mu}\right),t\in(0, T].
\end{equation}
Using the condition on $g$ given in  \eqref{a11} and using the equation \eqref{a12},   we get
\begin{align}\label{a16}
g(t, z_\epsilon(t))-g(t, z(t))&\leq L\left[  \frac{z_\epsilon(t)}{f(t, z_\epsilon(t))}-\frac{z(t)}{f(t, z(t))}\right] \nonumber\\
&\leq   L\epsilon\, E_{\mu}\left( 2L\left(  \Psi \left( t\right) -\Psi \left( 0\right) \right)^{\mu}\right), ~ t\in J.
\end{align}
Utilizing the inequality \eqref{a16} in the inequality \eqref{aa15}, we get
\begin{align*}
^H \mathcal{D}^{\mu,\,\nu\,;\, \Psi}_{0^+}Z_\epsilon(t)
&\geq g(t, z_\epsilon(t))-L\epsilon E_{\mu}\left( 2L\left(  \Psi \left( t\right) -\Psi \left( 0\right) \right)^{\mu}\right)+2L \epsilon E_{\mu}\left( 2L\left(  \Psi \left( t\right) -\Psi \left( 0\right) \right)^{\mu}\right)\\
&\geq g(t, z_\epsilon(t))+L\epsilon E_{\mu}\left( 2L\left(  \Psi \left( t\right) -\Psi \left( 0\right) \right)^{\mu}\right)\\
&>  g(t, z_\epsilon(t)).
\end{align*}
Therefore,
$$
^H \mathcal{D}^{\mu,\,\nu\,;\, \Psi}_{0^+}\left[\frac{z_\epsilon(t)}{f(t, z_\epsilon(t) )} \right] >  g(t, z_\epsilon(t)),~t\in(0, T].
$$
Now, from the inequality \eqref{a13}, we have 
$$
\left( \Psi \left( t\right) -\Psi \left( 0\right) \right)^{1-\xi }z_\epsilon(t)|_{t=0}>\left( \Psi \left( t\right) -\Psi \left( 0\right) \right)^{1-\xi }z(t)|_{t=0}
$$
and from the hypothesis 
$$
\left( \Psi \left( t\right) -\Psi \left( 0\right) \right)^{1-\xi }z(t)|_{t=0}\geq\left( \Psi \left( t\right) -\Psi \left( 0\right) \right)^{1-\xi }y(t)|_{t=0}.  
$$
Therefore,
\begin{equation}\label{aa11}
\left( \Psi \left( t\right) -\Psi \left( 0\right) \right)^{1-\xi }z_\epsilon(t)|_{t=0}>\left( \Psi \left( t\right) -\Psi \left( 0\right) \right)^{1-\xi }y(t)|_{t=0}.
\end{equation}
By applying Theorem \ref{tha4.1} with $z=z_\epsilon$, the condition \eqref{aa11} implies
\begin{equation}\label{aa12}
y \prec z_\epsilon~\text{in}~ C_{1-\xi ;\, \Psi }(J,\,\R).
\end{equation} 
Taking $\epsilon\rightarrow0$ in the equation \eqref{a12} and using the hypothesis (H1)(i), we have 
\begin{equation}\label{aa13}
\lim\limits_{\epsilon\rightarrow 0}z_\epsilon(t)=z(t).
\end{equation} 
Using \eqref{aa13}
\begin{align*}
\lim\limits_{\epsilon\rightarrow 0}\left\Vert z_\epsilon-z\right\Vert _{C_{1-\xi ;\,\Psi }\left(J,\,\R\right)  }=\lim\limits_{\epsilon\rightarrow 0}\left\lbrace \underset{t\in \left[ 0,T\right] 
}{\max } \left( \Psi \left( t\right) -\Psi \left( 0\right) \right)
^{1-\xi }\left\vert z_\epsilon(t)-z(t)\right| \right\rbrace =0.
\end{align*}
This gives $\lim\limits_{\epsilon\rightarrow 0}z_\epsilon=z$ in $ C_{1-\xi ;\, \Psi }(J,\,\R)$.  Therefore, the  inequality \eqref{aa12} reduces to $y\preceq z$ in $ C_{1-\xi ;\, \Psi }(J,\,\R)$.
\end{proof}

\section{Maximal and minimal solutions}
In this section, we will demonstrate the existence of maximal and minimal solutions for the hybrid FDEs \eqref{eqq1}-\eqref{eqq2} in  the weighted space $C_{1-\xi ;\, \Psi }(J,\,\R)$.
\begin{definition}
A solution $r$ of the hybrid FDEs \eqref{eqq1}-\eqref{eqq2} is said to be maximal solution if for any other solution $y$ to the hybrid FDEs \eqref{eqq1}-\eqref{eqq2}, we have  $y\preceq r$\, in\, $ C_{1-\xi ;\, \Psi }(J,\,\R)$.
\end{definition}
\begin{definition}
A solution $q$ of the hybrid FDEs \eqref{eqq1}-\eqref{eqq2} is said to be minimal solution if for any other solution $y$ to the hybrid FDEs \eqref{eqq1}-\eqref{eqq2}, we have  $q\preceq y $\, in\, $ C_{1-\xi ;\, \Psi }(J,\,\R)$.
\end{definition}

Let any $\epsilon>0$  and consider the following hybrid FDEs,
\begin{align}
&^H \mathcal{D}^{\mu,\,\nu\,;\, \Psi}_{0^+}\left[ \frac{y(t)}{ f(t, y(t))}\right] 
= g(t, y(t))+\epsilon,~a.e. ~t \in  (0,\,T],  ~\label{eqq11}\\
&\left( \Psi \left( t\right) -\Psi \left( 0\right) \right)^{1-\xi }y(t)|_{t=0}=y_{0}+\epsilon \in\R .\label{eqq12}
\end{align}

\begin{theorem}\label{tha5.1}
Assume that the hypotheses {\normalfont(H1)-(H2)}  and the condition  \eqref{aa1} hold. Then for every small number $\epsilon>0,$ the hybrid FDEs   \eqref{eqq11}-\eqref{eqq12} has a solution in $ C_{1-\xi ;\, \Psi }(J,\,\R)$.
\end{theorem}
\begin{proof}
By the hypothesis, we have
$$
L\left\lbrace \left| \frac{y_0\,}{f(0,y(0))}\right| + \frac{\left\| h\right\|_\infty \left( \Psi \left( T\right) -\Psi \left( 0\right) \right)^{\mu }  }{\Gamma(\mu+1)}\right\rbrace <1,
$$
then there exists  $\epsilon_0>0$ such that 
$$
L\left\lbrace \left| \frac{y_0+\epsilon\,}{f(0,y(0))}\right| + \frac{\left( \left\| h\right\|_\infty+\epsilon\right)  \left( \Psi \left( T\right) -\Psi \left( 0\right) \right)^{\mu }  }{\Gamma(\mu+1)}\right\rbrace <1, ~ 0<\epsilon\leq \epsilon_0.
$$
Following the similar steps as in the proof of Theorem \ref{tha3.2}  for the existence of a solution, one can complete the remaining part of the proof.
\end{proof}

\begin{theorem}\label{tha5.2}
Assume that the hypotheses {\normalfont(H1)-(H2)}  and the condition  \eqref{aa1} hold. Then for every  $\epsilon>0,$  the hybrid FDEs   \eqref{eqq1}-\eqref{eqq2} has a maximal  solution  in $ C_{1-\xi ;\, \Psi }(J,\,\R)$.
\end{theorem}
\begin{proof}
Let $\left\lbrace \epsilon_n\right\rbrace _{n=0}^\infty $ be a decreasing sequence of positive real numbers such that $\lim\limits_{n \rightarrow\infty }\epsilon_n=0 $ where  $\epsilon_0 $ is a positive real number satisfying the inequality 
\begin{equation}\label{a18}
L\left\lbrace \left| \frac{y_0+\epsilon_0\,}{f(0,y(0))}\right| + \frac{\left( \left\| h\right\|_\infty+\epsilon_0\right)  \left( \Psi \left( T\right) -\Psi \left( 0\right) \right)^{\mu }  }{\Gamma(\mu+1)}\right\rbrace <1.
\end{equation}
The existence of number $\epsilon_0$ can be achieved in the view of  the inequality \eqref{aa1}.
Since $\left\lbrace \epsilon_n\right\rbrace _{n=0}^\infty $ is a decreasing sequence, we have $\epsilon_n\leq \epsilon_0, n \in \N \cup \{0\}$ and one can verify that
\begin{equation*}
L\left\lbrace \left| \frac{y_0+\epsilon_n\,}{f(0,y(0))}\right| + \frac{\left( \left\| h\right\|_\infty+\epsilon_n\right)  \left( \Psi \left( T\right) -\Psi \left( 0\right) \right)^{\mu }  }{\Gamma(\mu+1)}\right\rbrace <1,~\text{for all}~n \in \N \cup \{0\}.
\end{equation*}
In the view of above condition,  Theorem \ref{tha5.1} guarantee the existence of solution $r(\cdot, \epsilon_n) \in  C_{1-\xi ;\, \Psi }(J,\,\R)$ of the hybrid FDEs
\begin{align}~\label{eqq13}
\begin{cases}
&^H \mathcal{D}^{\mu,\,\nu\,;\, \Psi}_{0^+}\left[ \frac{y(t)}{ f(t, y(t))}\right] 
= g(t, y(t))+\epsilon_n,~a.e. ~t \in  (0,\,T],  \\
&\left( \Psi \left( t\right) -\Psi \left( 0\right) \right)^{1-\xi }y(t)|_{t=0}=y_{0}+\epsilon_n \in\R. 
\end{cases}
\end{align}
From \eqref{eqq13}, it follows that
$$
^H \mathcal{D}^{\mu,\,\nu\,;\, \Psi}_{0^+}\left[ \frac{r(t, \epsilon_n)}{ f(t, r(t, \epsilon_n))}\right] 
> g(t, r(t, \epsilon_n)),~a.e. ~t \in  (0,\,T].
$$
Further, any solution $u\in  C_{1-\xi ;\, \Psi }(J,\,\R)$ of the hybrid FDEs   \eqref{eqq1}-\eqref{eqq2} satisfies
\begin{align*}
^H \mathcal{D}^{\mu,\,\nu\,;\, \Psi}_{0^+}\left[ \frac{u(t)}{ f(t, u(t))}\right] 
&\leq g(t, u(t)),~a.e. ~t \in  (0,\,T].
\end{align*}
By Theorem \ref{tha4.1}, the condition 
$$
\left( \Psi \left( t\right) -\Psi \left( 0\right) \right)^{1-\xi }u(t)|_{t=0}=y_{0}<y_{0}+\epsilon_n=\left( \Psi \left( t\right) -\Psi \left( 0\right) \right)^{1-\xi }r(t, \epsilon_n)|_{t=0}
$$ 
implies
\begin{equation}\label{a19}
u\prec r(\cdot, \epsilon_n),~\text{for all} ~n\in \N\cup\{0 \} ~\text{in}~ C_{1-\xi ;\, \Psi }(J,\,\R).
\end{equation}
Let $ r(t, \epsilon_1)$ and $ r(t, \epsilon_2)$ be the solutions of the hybrid FDEs  \eqref{eqq13}. Then, we have
\begin{align*}
^H \mathcal{D}^{\mu,\,\nu\,;\, \Psi}_{0^+}\left[ \frac{r(t, \epsilon_1)}{ f(t, r(t, \epsilon_1))}\right] 
&= g(t, r(t, \epsilon_1))+\epsilon_1\\
&> g(t, r(t, \epsilon_1))+\epsilon_2,
\end{align*}
$$
^H \mathcal{D}^{\mu,\,\nu\,;\, \Psi}_{0^+}\left[ \frac{r(t, \epsilon_2)}{ f(t, r(t, \epsilon_2))}\right] 
\leq  g(t, r(t, \epsilon_2))+\epsilon_2
$$
and
$$
\left( \Psi \left( t\right) -\Psi \left( 0\right) \right)^{1-\xi }r(t, \epsilon_1)|_{t=0}=y_{0}+\epsilon_1>y_{0}+\epsilon_2=\left( \Psi \left( t\right) -\Psi \left( 0\right) \right)^{1-\xi }r(t, \epsilon_2)|_{t=0}.
$$
Again by applying  Theorem \ref{tha4.1}, we have
$$
r(\cdot, \epsilon_1)\succ r(\cdot, \epsilon_2) ~\text{in}~ C_{1-\xi ;\, \Psi }(J,\,\R).
$$ 
Proceeding in this way, we obtain, $r(\cdot, \epsilon_n)$ is a bounded below decreasing sequence in $C_{1-\xi ;\, \Psi }(J,\,\R)$. Therefore, it is convergent. Let $ r \in  C_{1-\xi ;\, \Psi }(J,\,\R)$ such that 
$$\left\Vert r(\cdot, \epsilon_n)-r\right\Vert _{C_{1-\xi ;\,\Psi }\left(J,\,\R\right)  }\rightarrow 0~\text{ as}~n\rightarrow\infty.
$$
Then, we have
\begin{equation}\label{a20}
\left( \Psi \left( t\right) -\Psi \left( 0\right) \right)^{1-\xi }r(t)=\lim\limits_{n\rightarrow\infty} \left( \Psi \left( t\right) -\Psi \left( 0\right) \right)^{1-\xi }r(t, \epsilon_n),~t\in J.
\end{equation}
 We show that the convergence in  \eqref{a20} is uniform on $J$. For this it is enough to prove that the sequence $\{\left( \Psi \left( t\right) -\Psi \left( 0\right) \right)^{1-\xi }r(t, \epsilon_n)\}$ is equicontinious on $J$. Let $t_1, t_2\in J$ with $t_1>t_2$ be arbitrary. Then,
\begin{small}
\begin{align*}
&\left| \left( \Psi \left( t_1\right) -\Psi \left( 0\right) \right)^{1-\xi }r(t_1, \epsilon_n)-\left( \Psi \left( t_2\right) -\Psi \left( 0\right) \right)^{1-\xi }r(t_2, \epsilon_n)\right|\nonumber \\
&=\left| f(t_1,r(t_1, \epsilon_n))\left\lbrace  \frac{y_{0}+\epsilon_n}{f(0,r(0, \epsilon_n))}+\frac{ \left( \Psi \left( t_1\right) -\Psi \left( 0\right) \right)
^{1-\xi }}{\Gamma \left( \mu
\right) }\int_{0}^{t_1}\Psi'(s)(\Psi(t_1)-\Psi(s))^{\mu-1} \left(  g(s,r(s,\epsilon_n))+\epsilon_n\right)  \,ds\right\rbrace \right.\nonumber\\
&\left.- f(t_2,r(t_2, \epsilon_n))\left\lbrace  \frac{y_{0}+\epsilon_n}{f(0,r(0, \epsilon_n))}+\frac{ \left( \Psi \left( t_2\right) -\Psi \left( 0\right) \right)
^{1-\xi }}{\Gamma \left( \mu
\right) }\int_{0}^{t_2}\Psi'(s)(\Psi(t_2)-\Psi(s))^{\mu-1} \left(   g(s,r(s,\epsilon_n))+\epsilon_n\right)   \,ds\right\rbrace \right|\nonumber\\
&=\left| f(t_1,r(t_1, \epsilon_n))\left\lbrace \frac{y_{0}+\epsilon_n}{f(0,r(0, \epsilon_n))}+\frac{ \left( \Psi \left( t_1\right) -\Psi \left( 0\right) \right)
^{1-\xi }}{\Gamma \left( \mu
\right) }\int_{0}^{t_1}\Psi'(s)(\Psi(t_1)-\Psi(s))^{\mu-1} \left(  g(s,r(s,\epsilon_n))+\epsilon_n\right)  \,ds\right\rbrace \right.\nonumber\\
&\left.-f(t_2,r(t_2, \epsilon_n))\left\lbrace \frac{y_{0}+\epsilon_n}{f(0,r(0, \epsilon_n))}+\frac{ \left( \Psi \left( t_1\right) -\Psi \left( 0\right) \right)
^{1-\xi }}{\Gamma \left( \mu
\right) }\int_{0}^{t_1}\Psi'(s)(\Psi(t_1)-\Psi(s))^{\mu-1}\left(  g(s,r(s,\epsilon_n))+\epsilon_n\right)  \,ds\right\rbrace \right.\nonumber\\
&\left.+f(t_2,r(t_2, \epsilon_n))\left\lbrace  \frac{y_{0}+\epsilon_n}{f(0,r(0, \epsilon_n))}+\frac{ \left( \Psi \left( t_1\right) -\Psi \left( 0\right) \right)
^{1-\xi }}{\Gamma \left( \mu
\right) }\int_{0}^{t_1}\Psi'(s)(\Psi(t_1)-\Psi(s))^{\mu-1} \left(  g(s,r(s,\epsilon_n))+\epsilon_n\right)   \,ds\right\rbrace \right.\nonumber\\
&\left.- f(t_2,r(t_2, \epsilon_n))\left\lbrace  \frac{y_{0}+\epsilon_n}{f(0,r(0, \epsilon_n))}+\frac{ \left( \Psi \left( t_2\right) -\Psi \left( 0\right) \right)
^{1-\xi }}{\Gamma \left( \mu
\right) }\int_{0}^{t_2}\Psi'(s)(\Psi(t_2)-\Psi(s))^{\mu-1}\left(  g(s,r(s,\epsilon_n))+\epsilon_n\right) \,ds\right\rbrace \right| \nonumber \\
&\leq\left| f(t_1,r(t_1, \epsilon_n))-f(t_2,r(t_2, \epsilon_n))\right|\times\nonumber\\
&\qquad \left\lbrace \left|  \frac{y_{0}+\epsilon_n}{f(0,r(0,\epsilon_n))}\right| +\frac{ \left( \Psi \left( t_1\right) -\Psi \left( 0\right) \right)
^{1-\xi }}{\Gamma \left( \mu
\right) }\int_{0}^{t_1}\Psi'(s)(\Psi(t_1)-\Psi(s))^{\mu-1}\left(   \left|  g(s,r(s,\epsilon_n))\right| +\epsilon_n\right)   \,ds\right\rbrace\nonumber \\
& +\left| f(t_2,r(t_2, \epsilon_n))\right| \left\lbrace \frac{ \left( \Psi \left( t_1\right) -\Psi \left( 0\right) \right)
^{1-\xi }}{\Gamma \left( \mu
\right) }\int_{0}^{t_1}\Psi'(s)(\Psi(t_1)-\Psi(s))^{\mu-1}\left(   \left|  g(s,r(s,\epsilon_n))\right| +\epsilon_n\right)    \,ds \right.\nonumber\\
&\left.\qquad-\frac{ \left( \Psi \left( t_2\right) -\Psi \left( 0\right) \right)
^{1-\xi }}{\Gamma \left( \mu
\right) }\int_{0}^{t_2}\Psi'(s)(\Psi(t_2)-\Psi(s))^{\mu-1}\left(  \left|  g(s,r(s,\epsilon_n))\right| +\epsilon_n\right)  \,ds\right\rbrace .
\end{align*}
\end{small}
Using the hypothesis (H2), we have $ \left|  g(t,r(t,\epsilon_n))\right| +\epsilon_n\leq  \left\| h\right\|_\infty+\epsilon_n,\,t\in J$. Therefore, from the above inequality, we obtain 
 \begin{align}\label{a21}
 &\left| \left( \Psi \left( t_1\right) -\Psi \left( 0\right) \right)^{1-\xi }r(t_1, \epsilon_n)-\left( \Psi \left( t_2\right) -\Psi \left( 0\right) \right)^{1-\xi }r(t_2, \epsilon_n)\right|\nonumber\\
&\leq\left| f(t_1,r(t_1, \epsilon_n))-f(t_2,r(t_2, \epsilon_n))\right| \left\lbrace \left|  \frac{y_{0}+\epsilon_n}{f(0,r(0, \epsilon_n))}\right| +\frac{\left( \left\| h\right\|_\infty+\epsilon_n\right) \left( \Psi \left( t_1\right) -\Psi \left( 0\right) \right)
^{\mu+1-\xi }}{\Gamma \left( \mu+1
\right) }\right\rbrace \nonumber\\
& \qquad+K^* \left( \left\| h\right\|_\infty+\epsilon_n\right)  \left\lbrace  \frac{ \left( \Psi \left( t_1\right) -\Psi \left( 0\right) \right)^{\mu+1-\xi }}{\Gamma \left( \mu+1\right) }- \frac{\left( \Psi \left( t_2\right) -\Psi \left( 0\right) \right)^{\mu+1-\xi }}{\Gamma \left( \mu+1\right) }\right\rbrace ,
\end{align}
where $ K^*=\underset{(t, y)\in J\times[-R,\,R]}{\sup }\left| f(t,r(t, \epsilon_n))\right| $.
Since, $f$ is continuous on compact set $J\times[-R,\,R]$, we have
\begin{equation}\label{aa22}
\left| f(t_1,r(t_1, \epsilon_n))-f(t_2,r(t_2, \epsilon_n))\right| \rightarrow 0 ~\text{ as}~ \left|t_1- t_2\right|\rightarrow 0
\end{equation}
uniformly for all $n\in\N\cup \{0\}$. Using the condition \eqref{aa22} and continuity of $\Psi$ function in the inequality \eqref{a21}, we obtain
$$
\left| \left( \Psi \left( t_1\right) -\Psi \left( 0\right) \right)^{1-\xi }r(t_1, \epsilon_n)-\left( \Psi \left( t_2\right) -\Psi \left( 0\right) \right)^{1-\xi }r(t_2, \epsilon_n)\right|\rightarrow 0~\text{ as}~  \left|t_1- t_2\right|\rightarrow 0.
$$
This proves $\{\left( \Psi \left( t\right) -\Psi \left( 0\right) \right)^{1-\xi }r(t, \epsilon_n)\}$ is equicontinuous and hence 
$\left( \Psi \left( t\right) -\Psi \left( 0\right) \right)^{1-\xi }r(t, \epsilon_n)\rightarrow \left( \Psi \left( t\right) -\Psi \left( 0\right) \right)^{1-\xi }r(t)$ converges uniformly on $J$.      
Next, we prove  that $r\in {C_{1-\xi ;\,\Psi }\left(J,\,\R\right)  }$ is a solution of the hybrid FDEs \eqref{eqq1}-\eqref{eqq2}. Since $r(\cdot, \epsilon_n)\in {C_{1-\xi ;\,\Psi }\left(J,\,\R\right)  }$ is a solution of the hybrid FDEs  \eqref{eqq13}, we have 
$$
r(t, \epsilon_n)=f(t,r(t, \epsilon_n))\left\lbrace \frac{y_{0}+\epsilon_n}{f(0,r(0,\,\epsilon_n))}\left( \Psi \left( t\right) -\Psi \left( 0\right) \right)^{\xi-1 }+\mathcal{I}_{0^+}^{\mu\,;\, \Psi}\left( g(t,r(t, \epsilon_n))+\epsilon_n\right) \right\rbrace , ~ t\in (0,T].
$$
Using the continuity of functions $f$ and $g$, taking the limit as $n\rightarrow \infty$ in the above equation, we get
$$
r(t)=f(t,r(t))\left\lbrace \frac{y_{0}}{f(0, r(0))}\left( \Psi \left( t\right) -\Psi \left( 0\right) \right)^{\xi-1 }+\mathcal{I}_{0^+}^{\mu\,;\, \Psi}g(t,r(t))\right\rbrace , ~ t\in (0,T].
$$
Thus, $r(t)$ is a solution of the hybrid FDEs \eqref{eqq1}-\eqref{eqq2}. 
From the inequality \eqref{a19} it follows that 
$
u\preceq \lim\limits_{n\rightarrow \infty}r(\cdot, \epsilon_n)~ \text{in}~ {C_{1-\xi ;\,\Psi }\left(J,\,\R\right)  }
$. Therefore $
u\preceq r~ \text{in}~ {C_{1-\xi ;\,\Psi }\left(J,\,\R\right)  }
$.
This proves  $r$ is a maximal solution  the hybrid FDEs \eqref{eqq1}-\eqref{eqq2} in $C_{1-\xi ;\,\Psi }\left(J,\,\R\right)$. 
\end{proof}
\begin{rem}
The confirmation of the existence of minimal solution for the hybrid FDEs \eqref{eqq1}-\eqref{eqq2} one can finish on comparable lines.
\end{rem}

\section{Comparison Theorems}
\begin{theorem}\label{tha6.1}
Assume that the hypotheses {\normalfont(H1)-(H2)} and the condition  \eqref{aa1} hold. If there exists a function $u\in C_{1-\xi ;\, \Psi }(J,\,\R)$ such that
\begin{align}
&^H \mathcal{D}^{\mu,\,\nu\,;\, \Psi}_{0^+}\left[ \frac{u(t)}{ f(t, u(t))}\right] 
\leq g(t, u(t)),~a.e. ~t \in  (0,\,T],  ~\label{a22}\\
&\left( \Psi \left( t\right) -\Psi \left( 0\right) \right)^{1-\xi }u(t)|_{t=0}\leq y_0,\label{a23}
\end{align}
then
$$
u\preceq r~\text{in}~ C_{1-\xi ;\,\Psi }\left(J,\,\R\right),
$$
 where $r$ is a maximal solution of the hybrid FDEs \eqref{eqq1}-\eqref{eqq2}.
\end{theorem}
\begin{proof}
Let $\epsilon>0$ be arbitrary small. By Theorem \ref{tha5.2}, the hybrid FDEs \eqref{eqq11}-\eqref{eqq12} has a maximal solution $ r(\cdot, \epsilon) \in C_{1-\xi ;\,\Psi }\left(J,\,\R\right)$. Further, the  limit 
\begin{equation}\label{a24}
\left( \Psi \left( t\right) -\Psi \left( 0\right) \right)^{1-\xi }r(t)=\lim\limits_{\epsilon\rightarrow0} \left( \Psi \left( t\right) -\Psi \left( 0\right) \right)^{1-\xi }r(t, \epsilon)
\end{equation}
 is uniform on $J$, where  $r \in C_{1-\xi ;\,\Psi }\left(J,\,\R\right) $ is a  maximal solution of the hybrid FDEs \eqref{eqq1}-\eqref{eqq2}. As $ r(t, \epsilon)$ is a maximal solution  of the hybrid FDEs \eqref{eqq11}-\eqref{eqq12}. Therefore,
 \begin{align}
& ^H \mathcal{D}^{\mu,\,\nu\,;\, \Psi}_{0^+}\left[ \frac{ r(t, \epsilon)}{ f(t,  r(t, \epsilon))}\right] 
 = g(t,  r(t, \epsilon))+\epsilon,~a.e. ~t \in  (0,\,T],  ~\label{eqq21}\\
 &\left( \Psi \left( t\right) -\Psi \left( 0\right) \right)^{1-\xi } r(t, \epsilon)|_{t=0}=y_{0}+\epsilon \in\R .\label{eqq22}
 \end{align}
 From the above equations, we have 
  \begin{align}
 & ^H \mathcal{D}^{\mu,\,\nu\,;\, \Psi}_{0^+}\left[ \frac{ r(t, \epsilon)}{ f(t,  r(t, \epsilon))}\right] 
  > g(t,  r(t, \epsilon)),~a.e. ~t \in  (0,\,T],  ~\label{eqq31}\\
  &\left( \Psi \left( t\right) -\Psi \left( 0\right) \right)^{1-\xi } r(t, \epsilon)|_{t=0}>y_{0}\in\R .\label{eqq32}
  \end{align}
  From equations \eqref{a23} and \eqref{eqq32}, we have
  \begin{equation}\label{a25}
   \left( \Psi \left( t\right) -\Psi \left( 0\right) \right)^{1-\xi } r(t, \epsilon)|_{t=0}> \left( \Psi \left( t\right) -\Psi \left( 0\right) \right)^{1-\xi } u(t)|_{t=0}.
  \end{equation}
 Applying Theorem \ref{tha4.1}, from  the inequalities   \eqref{a22},  \eqref{eqq31} and  \eqref{a25}, we obtain
 $$
 u\prec r(\cdot,  \epsilon) ~\text{in}~ C_{1-\xi ;\,\Psi }\left(J,\,\R\right).
 $$
 Taking the limit  $\epsilon\rightarrow 0$ of above inequality and  using the equation  \eqref{a24}, we obtain
 $$
  u\preceq r~\text{in}~ C_{1-\xi ;\,\Psi }\left(J,\,\R\right).
  $$
\end{proof}
 
 The proof of the following theorem relating to the comparison of minimal and upper solution can be finished on the comparable lines of the proof of Theorem \ref{tha6.1}.
\begin{theorem}\label{tha6.2}
Assume that the hypotheses {\normalfont(H1)-(H2)} and the condition  \eqref{aa1} hold. If there exist a function $v\in C_{1-\xi ;\, \Psi }(J,\,\R)$ such that
\begin{align*}
&^H \mathcal{D}^{\mu,\,\nu\,;\, \Psi}_{0^+}\left[ \frac{v(t)}{ f(t, v(t))}\right] 
\geq g(t, v(t)),~a.e. ~t \in  (0,\,T],  \\
&\left( \Psi \left( t\right) -\Psi \left( 0\right) \right)^{1-\xi }v(t)|_{t=0}\geq y_0,
\end{align*}
then
$$
q\preceq v~\text{in}~ C_{1-\xi ;\,\Psi }\left(J,\,\R\right),
$$ where $q$ is a minimal solution of the hybrid FDEs \eqref{eqq1}-\eqref{eqq2}.
\end{theorem}

\section{Uniqueness of solution} 
In the accompanying theorem, we demonstrate the uniqueness of the solution to the hybrid FDEs \eqref{eqq1}-\eqref{eqq2} through the Theorem \ref{tha6.1}.

 \begin{theorem}\label{tha6.3}
 Assume that the hypotheses {\normalfont(H1)-(H2)} and the condition \eqref{aa1} hold. Also if there exists a function $G:J\times\R\rightarrow\R$ such that
 \begin{equation}\label{a26}
g(t, y_1)-g(t, y_2)\leq G\left( t,\, \frac{y_1}{f(t, y_1)}- \frac{y_2}{f(t, y_2)}\right),~a.e. ~t \in  (0,\,T], 
 \end{equation}
 for all $y_1, y_2\in \R$ with $y_1>y_2$.
 If the identically zero function is the only solution of the $\Psi$-Hilfer FDEs
\begin{align*}
&^H \mathcal{D}^{\mu,\,\nu\,;\, \Psi}_{0^+}m(t)
=G(t, m(t)),~t \in  (0,\,T],  \\
&\left( \Psi \left( t\right) -\Psi \left( 0\right) \right)^{1-\xi }m(t)|_{t=0}=0,
 \end{align*}
then, the hybrid FDEs \eqref{eqq1}-\eqref{eqq2}  has a unique solution.
 \end{theorem}
 \begin{proof}
 By Theorem \ref{tha3.2},  the hybrid FDEs \eqref{eqq1}-\eqref{eqq2}  has a  solution in $C_{1-\xi ;\,\Psi }\left(J,\,\R\right)$. Suppose that $u_1$ and $u_2$ are two solutions of the hybrid FDEs \eqref{eqq1}-\eqref{eqq2} with $u_1\succ u_2$ in $C_{1-\xi ;\,\Psi }\left(J,\,\R\right)$.  Define a function $m:(0,T]\rightarrow\R$ by 
 $$
 m(t)=\frac{u_1(t)}{f(t, u_1(t))}- \frac{u_2(t)}{f(t, u_2(t))},~ t\in (0,T]. 
 $$
 In the view of hypothesis (H1)(i), we obtain 
 \begin{equation}\label{aa4}
 m\succ 0 ~\mbox{in} ~C_{1-\xi ;\,\Psi }\left(J,\,\R\right).
 \end{equation} 
 Using the inequality \eqref{a26}, we get
 \begin{align*}
^H\mathcal{D}^{\mu,\,\nu\,;\, \Psi}_{0^+}m(t)
&=\,^H\mathcal{D}^{\mu,\,\nu\,;\, \Psi}_{0^+}\left[ \frac{u_1(t)}{f(t, u_1(t))}\right] -\, ^H\mathcal{D}^{\mu,\,\nu\,;\, \Psi}_{0^+}\left[ \frac{u_2(t)}{f(t, u_2(t))}\right] \\
&=g(t, u_1(t))-g(t, u_2(t))\\
&\leq G\left( t, \frac{u_1(t)}{f(t, u_1(t))}- \frac{u_2(t)}{f(t, u_2(t))}\right) \\
&= G(t, m(t)).
 \end{align*}
 Further,
 \begin{align*}
 \left( \Psi \left( t\right) -\Psi \left( 0\right) \right)^{1-\xi }m(t)|_{t=0}&=\left( \Psi \left( t\right) -\Psi \left( 0\right) \right)^{1-\xi }u_1(t)|_{t=0}-\left( \Psi \left( t\right) -\Psi \left( 0\right) \right)^{1-\xi }u_2(t)|_{t=0}=0.
 \end{align*}
 Therefore,
 \begin{align}\label{A}
 \begin{cases}
&^H\mathcal{D}^{\mu,\,\nu\,;\, \Psi}_{0^+}m(t)\leq G(t, m(t)), t\in(0,\,T],\\
 &\left( \Psi \left( t\right) -\Psi \left( 0\right) \right)^{1-\xi }m(t)|_{t=0}=0.
 \end{cases} 
 \end{align}
 By assumption the  identically zero function is the maximal solution of 
 \begin{align}\label{eqq4}
   \begin{cases}
 &^H \mathcal{D}^{\mu,\,\nu\,;\, \Psi}_{0^+}m(t)
 =G(t, m(t)),~t \in  (0,\,T],  \\
 &\left( \Psi \left( t\right) -\Psi \left( 0\right) \right)^{1-\xi }m(t)|_{t=0}=0.
 \end{cases}
  \end{align}
 Applying Theorem \ref{tha6.1} to the problems \eqref{A} and \eqref{eqq4}  with $f(t,\,x)=1$, $g=G$ and $y_0=0$,  we obtain
\begin{equation}\label{a27}
 m\preceq0~\text{in}~C_{1-\xi ;\,\Psi }\left(J,\,\R\right)
\end{equation}
The equation \eqref{a27} contradicts to the equation \eqref{aa4}. Thus, we must have $u_1=u_2$ in $C_{1-\xi ;\,\Psi }\left(J,\,\R\right)$.
 \end{proof}
 

\section*{Conclusion}
Existence and uniqueness of solution,  fractional differential inequalities and comparison results acquired in the present paper for $\Psi$-Hilfer hybrid FDEs includes the study of \cite{V. Lak1, Dhage1, Zhao,  Lak3}. Since the $\Psi$-Hilfer fractional derivative gives distinctive fractional derivative operators for different values of the parameters $\mu$, $\nu$ and the function $\Psi$, the outcomes of the present paper are also valid for the derivative operators listed in \cite{Vanterler1} as its special cases.  Further, the fractional integral inequalities and comparison results obtained in the setting of $\Psi$-Hilfer derivative can be utilized to analyze the various qualitative and quantitative properties of solutions for a different classes of $\Psi$-Hilfer hybrid FDEs.

\section*{Acknowledgment}
The first author  acknowledges the Science and Engineering Research Board (SERB), New Delhi, India for the Research Grant (Ref: File no. EEQ/2018/000407).


\end{document}